\documentclass[12pt]{article}
\usepackage[english]{babel}
\usepackage{latexsym}
\usepackage{amssymb}
\usepackage{amsmath}
\usepackage{amsthm}
\usepackage{amsfonts}
\usepackage{graphicx}

\def\'#1{\ifx#1i{\accent"13 \i}\else{\accent"13 #1}\fi}

\newtheorem{theorem}{Theorem}[section]

\title{The $4$-girth-thickness of the complete graph}
\author{Christian Rubio-Montiel\\ \emph{christian@cs.cinvestav.mx}
\\UMI LAFMIA 3175 CNRS at CINVESTAV-IPN\\Mexico City 07300, Mexico}

\begin{document}
\maketitle
\begin{abstract}
In this paper, we define the $4$-girth-thickness $\theta(4,G)$ of a graph $G$ as the minimum number of planar subgraphs of girth at least $4$ whose union is $G$. We prove that the $4$-girth-thickness of an arbitrary complete graph $K_n$, $\theta(4,K_n)$, is $\left\lceil \frac{n+2}{4}\right\rceil$ for $n\not=6,10$ and $\theta(4,K_6)=3$.
\end{abstract}



\section{Introduction}
A finite graph $G$ is \emph{planar} if it can be embedded in the plane without any two of its edges crossing. A planar graph of order $n$ and girth $g$ has size at most $\frac{g}{g-2}(n-2)$ (see \cite{MR2368647}), and an acyclic graph of order $n$ has size at most $n-1$, in this case, we define its girth as $\infty$. The \emph{thickness} $\theta(G)$ of a graph $G$ is the minimum number of planar subgraphs whose union is $G$; i.e. the minimum number of planar subgraphs into which the edges of $G$ can be partitioned.

The thickness was introduced by Tutte \cite{MR0157372} in 1963. Since then, exact results have been obtained when $G$ is a complete graph \cite{MR0460162,MR0164339,MR0186573}, a complete multipartite graph \cite{MR0158388,MR3243852,YY} or a hypercube \cite{MR0211901}. Also, some generalizations of the thickness for the complete graph $K_n$ have been studied such that the outerthickness $\theta_o$, defined similarly but with outerplanar instead of planar \cite{MR1100049}, and the $S$-thickness $\theta_S$, considering the thickness on a surfaces $S$ instead of the plane \cite{MR0245475}. See also the survey \cite{MR1617664}.

We define the \emph{$g$-girth-thickness} $\theta(g,G)$ of a graph $G$ as the minimum number of planar subgraphs of girth at least $g$ whose union is $G$. Note that the $3$-girth-thickness $\theta(3,G)$ is the usual thickness and the $\infty$-girth-thickness $\theta(\infty,G)$ is the \emph{arboricity number}, i.e.  the minimum number of acyclic subgraphs into which $E(G)$ can be partitioned. In this paper, we obtain the $4$-girth-thickness of an arbitrary complete graph of order $n\not=10$.


\section{The exact value of $\theta(4,K_n)$ for $n\not=10$}\label{Section2}

Since the complete graph $K_n$ has size $\binom{n}{2}$ and a planar graph of order $n$ and girth at least $4$ has size at most $2(n-2)$ for $n\geq 3$ and $n-1$ for $n\in\{1,2\}$ then the $4$-girth-thickness of $K_n$ is at least \[\left\lceil \frac{n(n-1)}{2(2n-4)}\right\rceil =\left\lceil \frac{n+1}{4}+\frac{1}{2n-4}\right\rceil =\left\lceil \frac{n+2}{4}\right\rceil\] for $n\geq 3$ and also $\left\lceil \frac{n+2}{4}\right\rceil$ for $n\in\{1,2\}$, we have the following theorem.

\begin{theorem} \label{teo1}
The $4$-girth-thickness $\theta(4,K_n)$ of $K_n$ equals $\left\lceil \frac{n+2}{4}\right\rceil$ for $n\not=6,10$ and $\theta(4,K_6)=3$.
\end{theorem}

\begin{proof}
Figure \ref{Fig1} displays equality for $n\leq 5$.

\begin{figure}[!htbp]
\begin{center}
\includegraphics[scale=0.7]{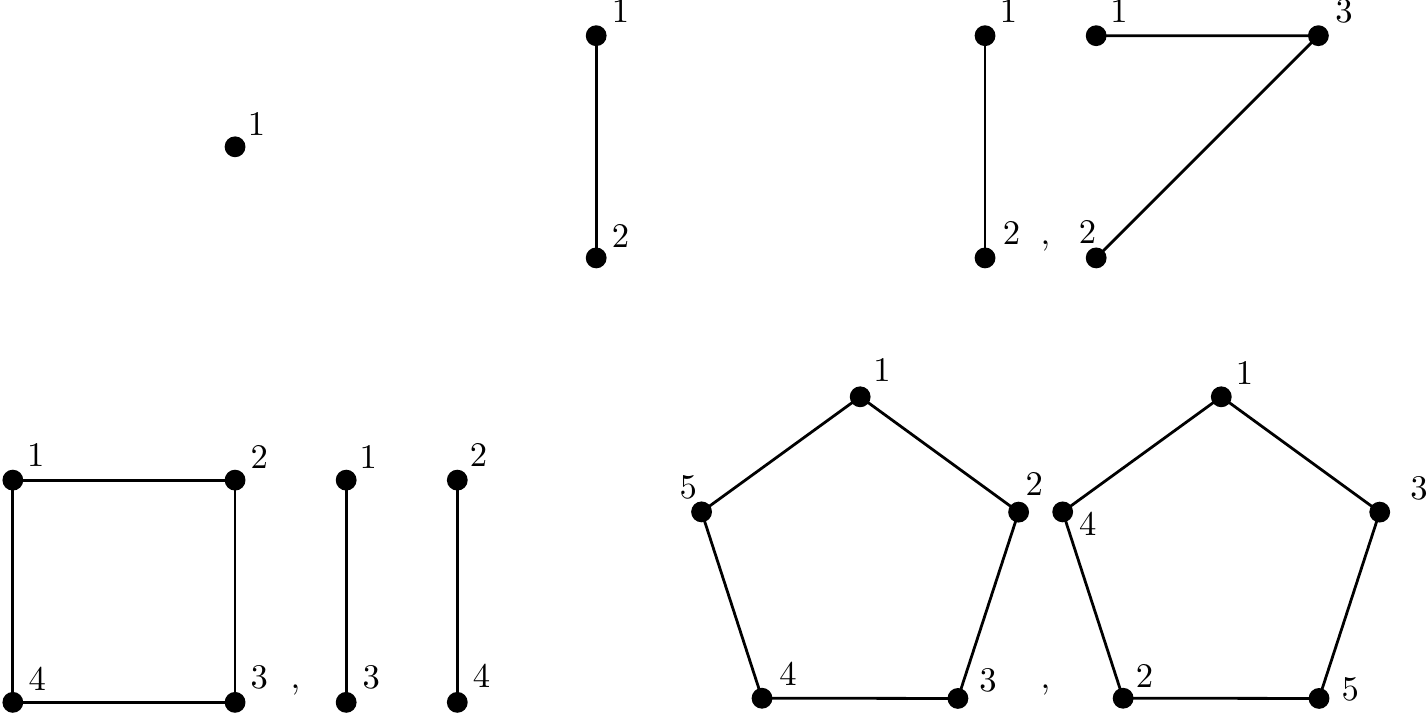}
\caption{\label{Fig1} $\theta(4,K_n)=\left\lceil \frac{n+2}{4}\right\rceil$ for $n=1,2,3,4,5$.}
\end{center}
\end{figure}

To prove that $\theta(4,K_6)=3>\left\lceil \frac{6+2}{4}\right\rceil=2$, suppose that $\theta(4,K_6)=2$. This partition define an edge coloring of $K_6$ with two colors. By Ramsey's Theorem, some part contains a triangle obtaining a contradiction for the girth $4$. Figure \ref{Fig2} shows a partition of $K_6$ into tree planar subgraphs of girth at least $4$.

\begin{figure}[!htbp]
\begin{center}
\includegraphics[scale=0.7]{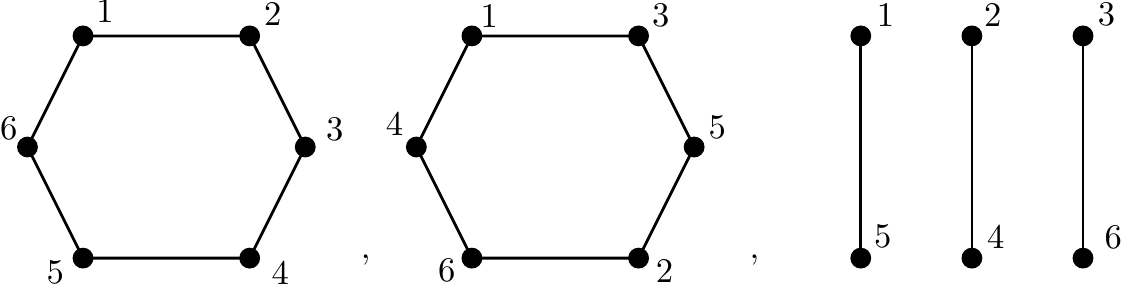}
\caption{\label{Fig2} $\theta(4,K_6)=3$.}
\end{center}
\end{figure}

For the remainder of this proof, we need to distinguish four cases, namely, when $n=4k-1$, $n=4k$, $n=4k+1$ and $n=4k+2$ for $k\geq2$. Note that in each case, the lower bound of the $4$-girth thickness require at least $k+1$ elements. To prove our theorem, we exhibit a decomposition of $K_{4k}$ into $k+1$ planar graphs of girth at least $4$. The other three cases are based in this decomposition. The case of $n=4k-1$ follows because $K_{4k-1}$ is a subgraph of $K_{4k}$. For the case of $n=4k+2$, we add two vertices and some edges to the decomposition obtained in the case of $n=4k$. The last case follows because $K_{4k+1}$ is a subgraph of $K_{4k+2}$. In the proof, all sums are taken modulo $2k$.

\begin{enumerate}
\item Case $n=4k$. It is well-known that a complete graph of even order contains a cyclic factorization of Hamiltonian paths, see \cite{MR2450569}. Let $G$ be a subgraph of $K_{4k}$ isomorphic to $K_{2k}$. Label its vertex set $V(G)$ as $\{v_1,v_2,\dots,v_{2k}\}$. Let $\mathcal{F}_1$ be the Hamiltonian path with edges \[v_1v_2,v_2v_{2k},v_{2k}v_3,v_3v_{2k-1},\dots,v_{k+2}v_{k+1}.\]
Let $\mathcal{F}_i$ be the Hamiltonian path with edges \[v_{1+i-1}v_{2+i-1},v_{2+i-1}v_{2k+i-1},v_{2k+i-1}v_{3+i-1},v_{3+i-1}v_{2k-1+i-1},\dots,v_{k+2+i-1}v_{k+1+i-1},\]
where $i\in\{2,3,\dots,k\}$.

Such factorization of $G$ is the partition $\{E(\mathcal{F}_1),E(\mathcal{F}_2),\dots,E(\mathcal{F}_k)\}$. We remark that the center of $\mathcal{F}_i$ has the edge $e=v_{i+\left\lceil \frac{k}{2}\right\rceil}v_{i+\left\lceil \frac{3k}{2}\right\rceil}$, see Figure \ref{Fig3}.
\begin{figure}[!htbp]
\begin{center}
\includegraphics{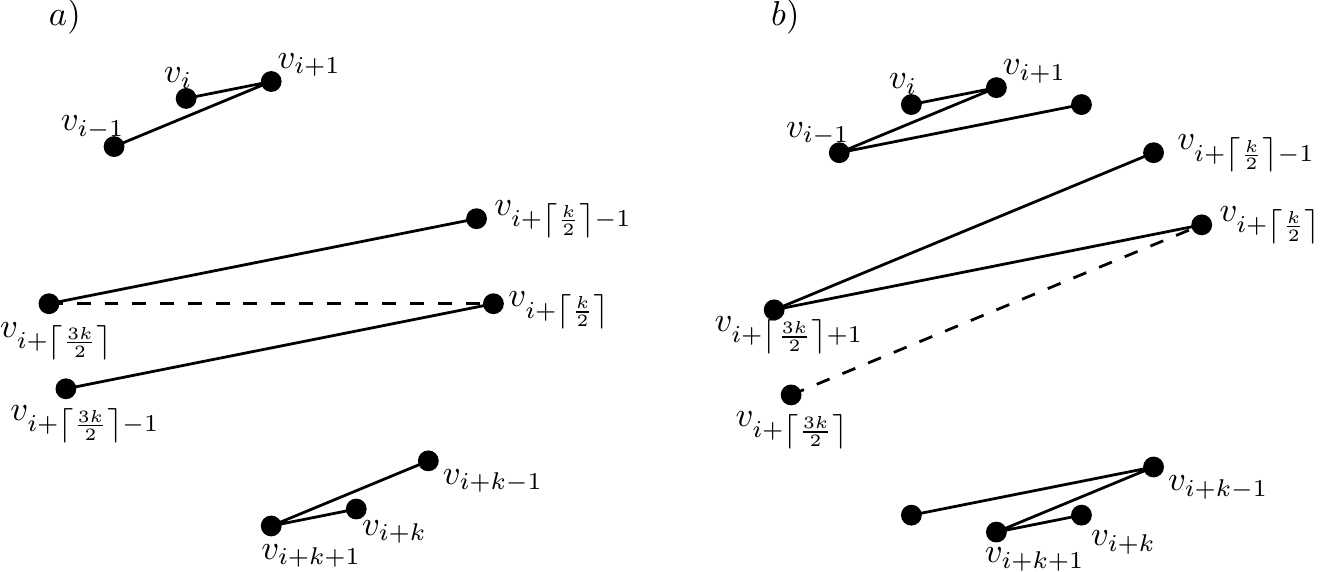}
\caption{\label{Fig3} The Hamiltonian path $\mathcal{F}_i$: Left $a)$: The dashed edge $e$ for $k$ odd. Right $b)$ The dashed edge $e$ for $k$ even.}
\end{center}
\end{figure}

Now, consider the complete subgraph $G'$ of $K_{4k}$ such that $G'=K_{4k}\setminus V(G)$. Label its vertex set $V(G')$ as $\{v'_1,v'_2,\dots,v'_{2k}\}$ and consider the factorization, similarly as before, $\{E(\mathcal{F}'_1),E(\mathcal{F}'_2),\dots,E(\mathcal{F}'_k)\}$ where $\mathcal{F}'_i$ is the Hamiltonian path with edges \[v'_{i}v'_{i+1},v'_{i+1}v'_{i-1},v'_{i-1}v'_{i+2},v'_{i+2}v'_{i-2},\dots,v'_{i+k+1}v'_{i+k},\]
where $i\in\{1,2,\dots,k\}$.

Next, we construct the planar subgraphs $G_1$, $G_2$,...,$G_{k-1}$ and $G_k$ of girth $4$, order $4k$ and size $8k-4$ (observe that $2(4k-2)=8k-4$), and also the matching $G_{k+1}$, as follows. Let $G_i$ be a spanning subgraph of $K_{4k}$ with edges $E(\mathcal{F}_i)\cup E(\mathcal{F}'_i)$ and \[v_{i}v'_{i+1},v'_{i}v_{i+1}, v_{i+1}v'_{i-1},v'_{i+1}v_{i-1},
v_{i-1}v'_{i+2},v'_{i-1}v_{i+2},\dots,
v_{i+k+1}v'_{i+k},v'_{i+k+1}v_{i+k}\]
where $i\in\{1,2,\dots,k\}$; and let $G_{k+1}$ be a perfect matching with edges $v_jv'_j$ for $j\in\{1,2,\dots,2k\}$. Figure \ref{Fig4} shows $G_i$ is a planar graph of girth at least $4$.
\begin{figure}[!htbp]
\begin{center}
\includegraphics{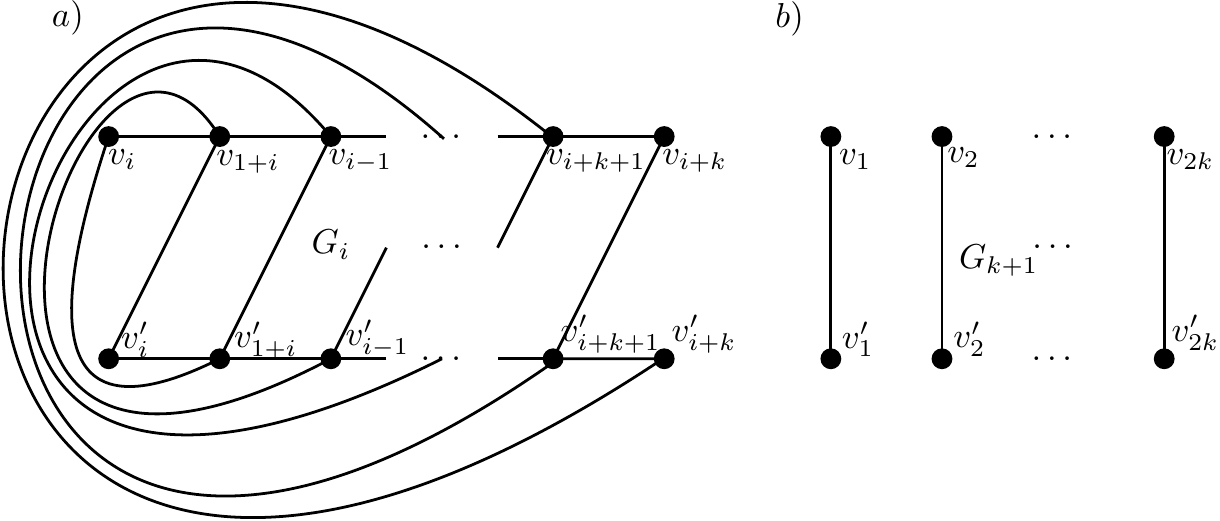}
\caption{\label{Fig4} Left $a)$: The graph $G_i$ for any $i\in\{1,2,\dots,k\}$. Right $b)$ The graph $G_{k+1}$.}
\end{center}
\end{figure}

To verify that $K_{4k}=\underset{i=1}{\overset{k+1}{\bigcup}}G_{i}$: 1) If the edge $v_{i_1}v_{i_2}$ of $G$ belongs to the factor $\mathcal{F}_i$ then $v_{i_1}v_{i_2}$ belongs to $G_i$. If the edge is primed, belongs to $G'_i$. 2) The edge $v_{i_1}v'_{i_2}$ belongs to $G_{k+1}$ if and only if $i_1=i_2$, otherwise it belongs to the same graph $G_i$ as $v_{i_1}v_{i_2}$. Similarly in the case of $v'_{i_1}v_{i_2}$ and the result follows.

\item Case $n=4k-1$. Since $K_{4k-1}\subset K_{4k}$, we have \[k+1\leq \theta(4,K_{4k-1})\leq\theta(4,K_{4k})\leq k+1.\]

\item Case $n=4k+2$ (for $k\not= 2$). Let $\{G_1,\dots,G_{k+1}\}$ be the planar decomposition of $K_{4k}$ constructed in the Case 1. We will add the two new vertices $x$ and $y$ to every planar subgraph $G_i$, when $1\leq i \leq k+1$, and we will add $4$ edges to each $G_i$, when $1\leq i \leq k$, and $4k+1$ edges to $G_{k+1}$ such that the resulting new subgraphs of $K_{4k+2}$ will be planar. Note that $\binom{4k}{2}+4k+4k+1=\binom{4k+2}{2}$.

To begin with, we define the graph $H_{k+1}$ adding the vertices $x$ and $y$ to the planar subgraph $G_{k+1}$ and the $4k+1$ edges \[\{xy,xv_1,xv'_2,xv_3,xv'_4,\dots,xv_{2k-1},xv'_{2k},yv'_1,yv_2,yv'_3,yv_4,\dots,yv'_{2k-1},yv_{2k}\}.\]
The graph $H_{k+1}$ has girth $4$, see Figure \ref{Fig5}.
\begin{figure}[!htbp]
\begin{center}
\includegraphics{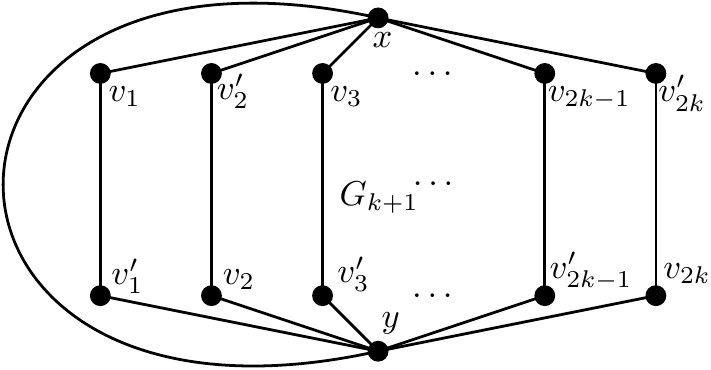}
\caption{\label{Fig5} The graph $H_{k+1}$.}
\end{center}
\end{figure}

In the following, for $1\leq i\leq k$, by adding vertices $x$ and $y$ to $G_i$ and adding $4$ edges to $G_i$, we will get a new planar graph $H_i$ such that $\{H_1,\dots,H_{k+1}\}$ is a planar decomposition of $K_{4k+2}$ such that the girth of every element is $4$. To achieve it, the given edges to the graph $H_i$ will be $v'_jx,xv_{j-1},v_jy,yv'_{j-1}$, for some odd $j\in\{1,3,\dots,2k-1\}$.

According to the parity of $k$, we have two cases:
\begin{itemize}
\item Suppose $k$ odd. For odd $i\in\{1,2,\dots,k\}$, we define the graph $H_{i}$ adding the vertices $x$ and $y$ to the planar subgraph $G_{i}$ and the $4$ edges \[\{xv'_{i+\left\lceil \frac{3k}{2}\right\rceil-1},xv_{i+\left\lceil \frac{3k}{2}\right\rceil},yv_{i+\left\lceil \frac{3k}{2}\right\rceil-1},yv'_{i+\left\lceil \frac{3k}{2}\right\rceil}\}\]
when $\left\lceil \frac{k}{2}\right\rceil$ is even, otherwise
\[\{yv'_{i+\left\lceil \frac{3k}{2}\right\rceil-1},yv_{i+\left\lceil \frac{3k}{2}\right\rceil},xv_{i+\left\lceil \frac{3k}{2}\right\rceil-1},xv'_{i+\left\lceil \frac{3k}{2}\right\rceil}\}.\]
Additionally, for even $i\in\{1,2,\dots,k\}$, we define the graph $H_{i}$ adding the vertices $x$ and $y$ to the planar subgraph $G_{i}$ and the $4$ edges \[\{xv'_{i+\left\lceil \frac{k}{2}\right\rceil-1},xv_{i+\left\lceil \frac{k}{2}\right\rceil},yv_{i+\left\lceil \frac{k}{2}\right\rceil-1},yv'_{i+\left\lceil \frac{k}{2}\right\rceil}\}\]
when $\left\lceil \frac{k}{2}\right\rceil$ is even, otherwise
\[\{yv'_{i+\left\lceil \frac{k}{2}\right\rceil-1},yv_{i+\left\lceil \frac{k}{2}\right\rceil},xv_{i+\left\lceil \frac{k}{2}\right\rceil-1},xv'_{i+\left\lceil \frac{k}{2}\right\rceil}\}.\]
Note that the graph $H_{i}$ has girth $4$ for all $i$, see Figure \ref{Fig6}.
\begin{figure}[!htbp]
\begin{center}
\includegraphics{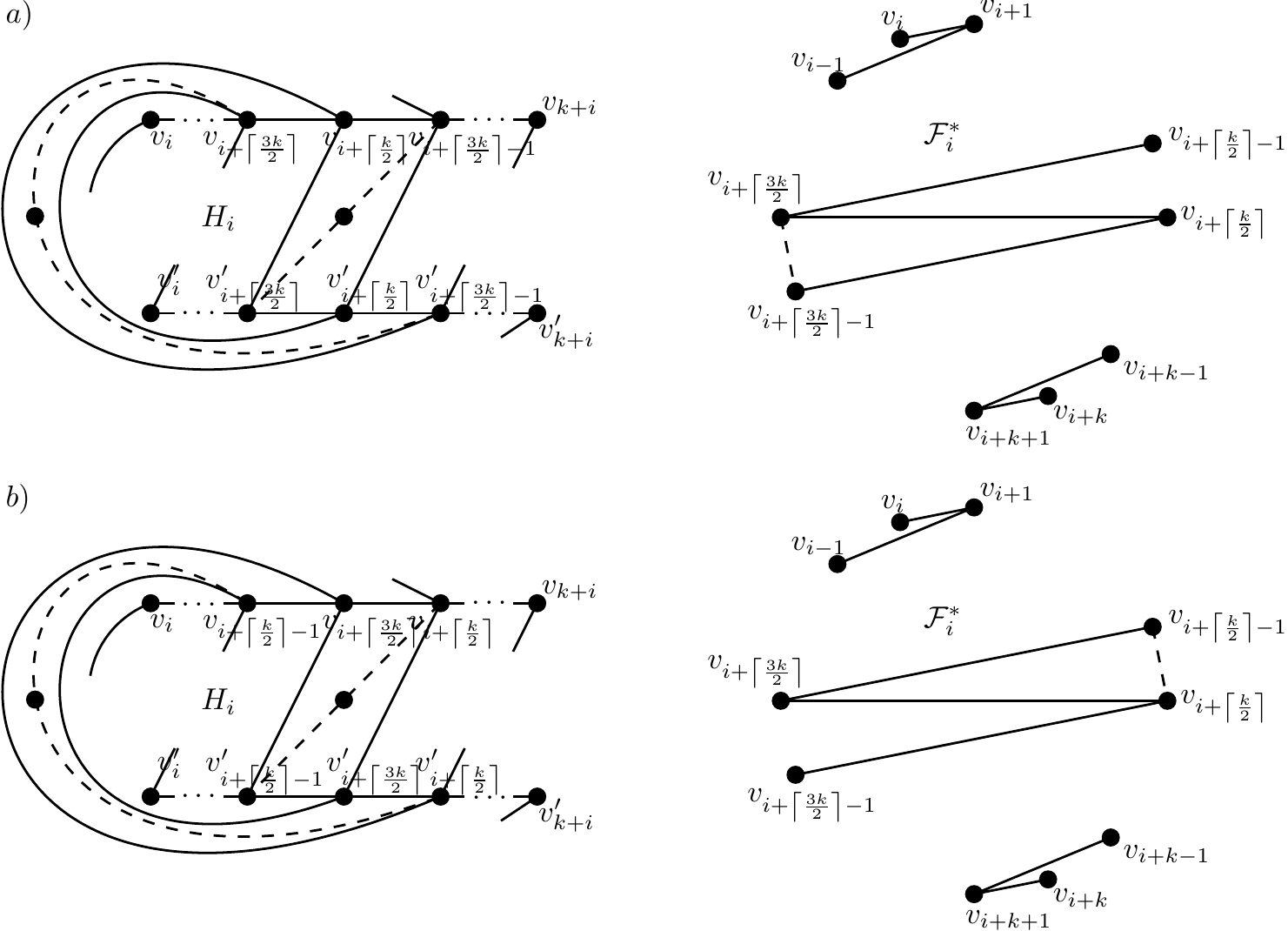}
\caption{\label{Fig6}The graph $H_{i}$ when $k$ is odd and its auxiliary graph $\mathcal{F}^{*}_i$.  Above $a)$ When $i$ is odd. Botton $b)$ When $i$ is even.}
\end{center}
\end{figure}

\item Suppose $k$ even. Similarly that the previous case, for odd $i\in\{1,2,\dots,k\}$, we define the graph $H_{i}$ adding the vertices $x$ and $y$ to the planar subgraph $G_{i}$ and the $4$ edges \[\{xv_{i+\left\lceil \frac{3k}{2}\right\rceil+1},xv'_{i+\left\lceil \frac{3k}{2}\right\rceil},yv'_{i+\left\lceil \frac{3k}{2}\right\rceil+1},yv_{i+\left\lceil \frac{3k}{2}\right\rceil}\}\]
when $\left\lceil \frac{k}{2}\right\rceil$ is even, otherwise
\[\{yv_{i+\left\lceil \frac{3k}{2}\right\rceil+1},yv'_{i+\left\lceil \frac{3k}{2}\right\rceil},xv'_{i+\left\lceil \frac{3k}{2}\right\rceil+1},xv_{i+\left\lceil \frac{3k}{2}\right\rceil}\}.\]
On the other hand, for even $i\in\{1,2,\dots,k\}$, we define the graph $H_{i}$ adding the vertices $x$ and $y$ to the planar subgraph $G_{i}$ and the $4$ edges \[\{xv_{i+\left\lceil \frac{k}{2}\right\rceil},xv'_{i+\left\lceil \frac{k}{2}\right\rceil-1},yv'_{i+\left\lceil \frac{k}{2}\right\rceil},yv_{i+\left\lceil \frac{k}{2}\right\rceil-1}\}\]
when $\left\lceil \frac{k}{2}\right\rceil$ is even, otherwise
\[\{yv_{i+\left\lceil \frac{k}{2}\right\rceil},yv'_{i+\left\lceil \frac{k}{2}\right\rceil-1},xv'_{i+\left\lceil \frac{k}{2}\right\rceil},xv_{i+\left\lceil \frac{k}{2}\right\rceil-1}\}.\]
Note that the graph $H_{i}$ has girth $4$ for all $i$, see Figure \ref{Fig7}.
\begin{figure}[!htbp]
\begin{center}
\includegraphics{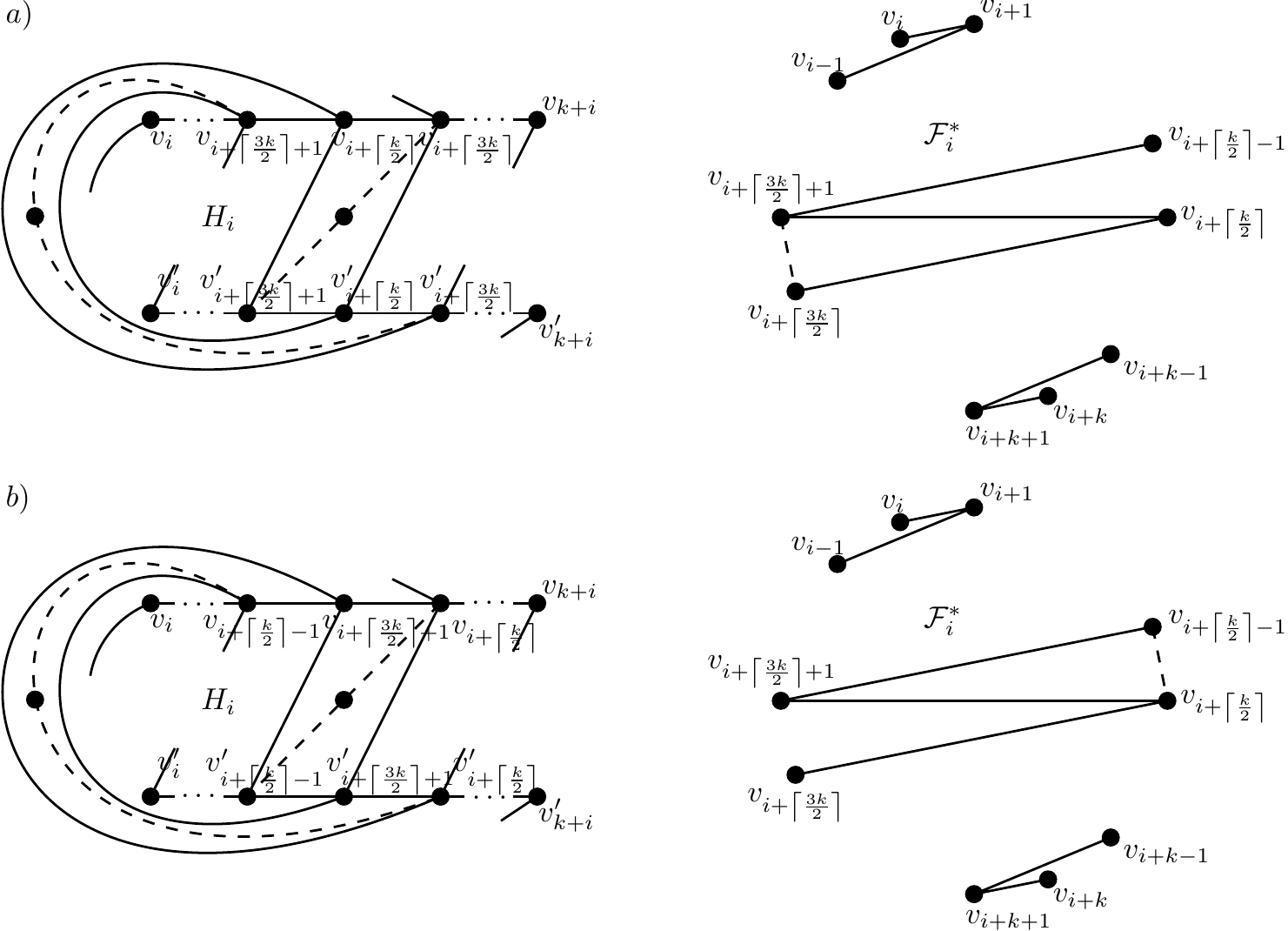}
\caption{\label{Fig7}The graph $H_{i}$ when $k$ is even and its auxiliary graph $\mathcal{F}^{*}_i$.  Above $a)$ When $i$ is odd. Botton $b)$ When $i$ is even.}
\end{center}
\end{figure}
\end{itemize}

In order to verify that each edge of the set \[\{xv'_1,xv_2,xv'_3,xv_3,\dots,xv'_{2k-1},xv_{2k},yv_1,yv'_2,yv_3,yv'_3,\dots,yv_{2k-1},yv'_{2k}\}.\]
is in exactly one subgraph $H_i$, for $i\in\{1,\dots,k\}$, we obtain the unicyclic graph $\mathcal{F}^{*}_i$ identifying $v_j$ and $v'_j$ resulting in $v_j$; identifying $x$ and $y$ resulting in a vertex which is contracted with one of its neighbours. The resulting edge, in dashed, is showed in Figures \ref{Fig6} and \ref{Fig7}. The set of those edges are a perfect matching of $K_{2k}$ proving that the added two paths of length 2 in $G_i$ have end vertices $v_j$ and $v'_{j-1}$, and the other $v'_j$ and $v_{j-1}$. The election of the label of the center vertex is such that one path is $v_{even}xv'_{odd}$ and $v'_{even}yv_{odd}$ and the result follows.
\item Case $n=4k+1$ (for $k\not=2$). Since $K_{4k+1}\subset K_{4k+2}$, we have \[k+1\leq \theta(4,K_{4k+1})\leq\theta(4,K_{4k+2})\leq k+1.\]
\end{enumerate}
For $k=2$, Figure \ref{Fig8} displays a decomposition of three planar graphs of girth at least $4$ proving that $\theta(4,K_9)=\left\lceil \frac{9+2}{4}\right\rceil=3$.
\begin{figure}[!htbp]
\begin{center}
\includegraphics[scale=0.7]{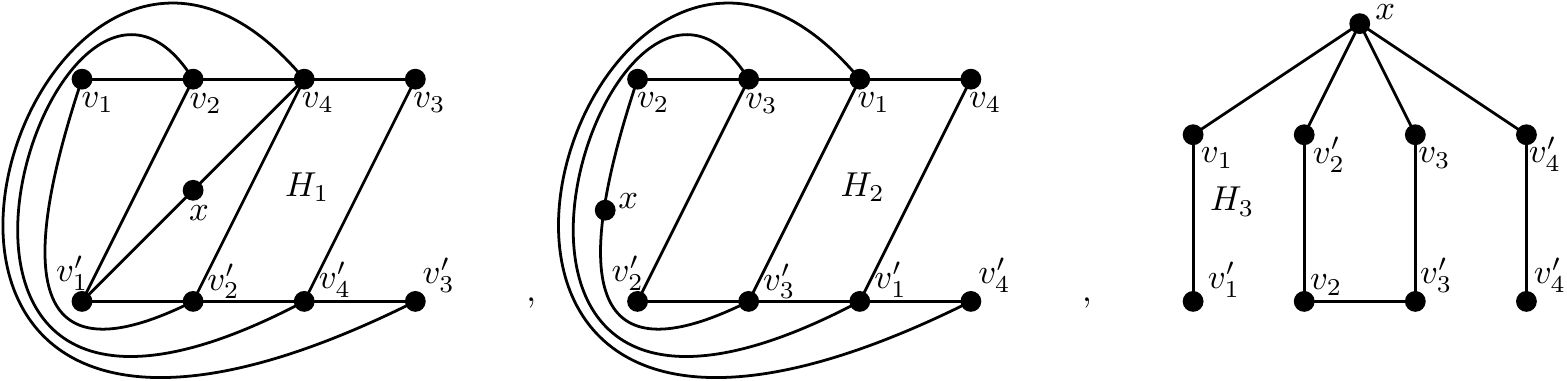}
\caption{\label{Fig8} A planar decomposition of $K_9$ into three subgraphs of girth $4$ and $5$.}
\end{center}
\end{figure}

By the four cases, the theorem follows.
\end{proof}

About the case of $K_{10}$, it follows $3\leq\theta(4,K_{10})\leq4$. We conjecture that $\theta(4,K_{10})=4$.


\end{document}